\documentclass[11pt]{amsart}   %%%%%%%Please do not change
\newtheorem{theo}{Theorem}[section]
\newtheorem{lemma}[theo]{Lemma}
\newtheorem{proposition}[theo]{Proposition}
\newtheorem{cor}[theo]{Corollary}
\theoremstyle{definition}

\theoremstyle{remark}
\newtheorem{remark}[theo]{Remark}

\begin{document}
%%%%%%%%%%%%%%%%%%%%%%%%%%%%%%%%%%%%%%%%%%%%%%%%%%%%%%%%%%%%
%%%%%%%%%%%%%%%%%%%%%%%%%%%%%%%%%%%%%%%%%%%%%%%%%%%%%%%%%%%%
% This a placeholder for the TOPLOGY PROCEEDINGS logo %%%%%%
%\noindent                                             %%%%%%
%\begin{picture}(150,36)                               %%%%%%
%\put(5,20){\tiny{Submitted to}}                       %%%%%%
%\put(5,7){\textbf{Topology Proceedings}}              %%%%%%
%\put(0,0){\framebox(140,34)}                          %%%%%%
%\put(2,2){\framebox(136,30)}                          %%%%%%
%\end{picture}                                         %%%%%%
%%%%%%%%%%%%%%%%%%%%%%%%%%%%%%%%%%%%%%%%%%%%%%%%%%%%%%%%%%%%
%%%%%%%%%%%%%%%%%%%%%%%%%%%%%%%%%%%%%%%%%%%%%%%%%%%%%%%%%%%%

\vspace{0.5in}

\title{ Itzkowitz's problem for groups 
of finite exponent}

%    Information for first author:
\author{A. Bareche}
\address{Universit\'e de Rouen, D\'epartement de Math\'ematiques,
UMR CNRS 6085, Avenue de l'Universit\'e, BP.12, F76801
Saint-\'Etienne-du-Rouvray, France}
%    Current address (if needed): 
%\curraddr{}
\email{aicha.bareche@etu.univ-rouen.fr}

%    Information for second author (if needed): 
\author{A. Bouziad}
\address{Universit\'e de Rouen, D\'epartement de Math\'ematiques,
UMR CNRS 6085, Avenue de l'Universit\'e, BP.12, F76801
Saint-\'Etienne-du-Rouvray, France
}
\email{ahmed.bouziad@univ-rouen.fr}

%    General info


\begin{abstract} Itzkowitz's problem asks whether every topological group $G$ has equal left and right uniform structures provided that 
bounded left uniformly continuous real-valued function on $G$
are right uniformly continuous. This paper provides a positive answer to this problem
if $G$ is of bounded exponent or, more generally,
if there exist an integer  $p\geq 2$ and a nonempty
open set $U\subset G$ such that the power map $U\ni g\to g^p\in G$
is left (or right) uniformly continuous. This also resolves the problem
for periodic  groups which are Baire spaces. 
\end{abstract}

\maketitle

\section{Introduction}
Let  $G$ be a Hausdorff topological group, $e$ its unit and  ${\mathcal{N}}(e)$ 
the set of all neighborhoods  of $e$ in $G$. The left uniformity ${\mathcal U}_l$ of $G$ has
as a basis of entourages the sets of the form $\{(x,y)\in G\times G: x^{-1}y\in V\}$,
where $V$ is a member of ${\mathcal N}(e)$ (we write the law of $G$ multiplicatively).
The right uniformity ${\mathcal U}_r$ is obtained by writing $xy^{-1}$ in place
of $x^{-1}y$ (in the form above). The group
$G$ is said to be balanced, or SIN (for small invariant neighborhoods), if  the
two uniformities ${\mathcal U}_l$ and ${\mathcal U}_r$ coincide. As it is well
known,
 compact groups and (obviously)  Abelian topological groups are SIN. A
topological
group $G$ is called functionally balanced or FSIN, if every
bounded left uniformly continuous real-valued function
on $G$ is right uniformly continuous. Here, a function $f: G\to\mathbb R$
is left uniformly continuous if $f$ is uniformly continuous
when $G$ is equipped with its left uniformity and the reals
with the usual metric. Right uniform continuity is defined similarly. Since
the inversion on $G$ switches the left and right uniformities, the
alternative right-left definition in the FSIN property leads to the same
thing. \par 
Obviously,
every SIN group is an FSIN group, but it is still unknown if  the converse holds for every
topological group. This problem, called Itzkowitz's problem after the work
\cite{IT1},
has received several positive answers in the past; the reader is referred to
\cite{BT1} and the references therein for more information
and related questions about the problem. Let us just recall that it  has received the attention
of many authors and  has been solved for two
notably different classes including respectively  locally compact groups and
locally connected groups (see for instance \cite{BT2, HER, IT2, MIL,PR1}).
It is clear that if $G$ is locally compact
then for every positive integer $p$, the power map $\phi_p$, defined by $G\ni
x\to x^p\in G$,
is left uniformly
continuous ``locally'', that is, when it is restricted to some neighborhood of the unit of $G$ (just
take a compact neighborhood) and $G$ is equipped with the left uniformity. In
particular, the class of groups for
which there
is $p\geq 2$ such that the power map $\phi_p$ is ``locally'' left uniformly continuous, 
includes all locally compact groups. It  includes also all groups of finite exponent; where  a group
$G$ is said to be of finite exponent if for some $p\geq 1$, $x^p=e$ for all $x\in G$.
In connection with this, it should be noted, as it is easy to see, that
a group $G$ is SIN if and only if the power map $\phi_2$ is left uniformly
continuous. This should
be compared to the well-known fact that $G$ is SIN if and only if the product
map $(x,y)\to xy$ is left uniformly continuous; see  \cite{ROE}.\par 
The main result of this note 
is as follows: The equality FSIN=SIN holds in the class of topological groups $G$
for which there are a neighborhood $V$ of the unit and an integer $p\geq 2$
such that the map $\phi_p: V\to G$ is left uniformly continuous.
This extends the locally compact case and allows to give an affirmative answer
to Itzkovitz's question for
periodic topological groups which are of the second category.

\section{Two lemmas}

In what follows, $G$ is a fixed Hausdorff topological group. A subset
$A$ of $G$ is called {\it left $V$-thin} in $G$, where 
$V\in{\mathcal N}(e)$, if the set $\cap_{a\in A}a^{-1}Va$
is a neighborhood of $e$ in $G$. If $A$ is left $V$-thin in $G$
for every $V\in{\mathcal N}(e)$, then $A$ is said
to be {\it left thin} in $G$. The concept of ``right $V$-thin'' is defined similarly.
Note that the group $G$ is SIN if and only if
$G$ is left thin in itself. The set $A$
is said to be {\it left} ({\it right}) {\it $V$-discrete} in $G$ if $a=b$ whenever $a\in bV$ ($a\in Vb$) and $a,b\in A$. The set $A$ is
said to be {\it Roelcke-discrete}
if  there is $V\in{\mathcal N}(e)$ such
that $a=b$ whenever $a,b\in A$ and   $a\in VbV$.
Note that this means that $A$ is a uniformly discrete subset of $G$ (with respect to $V$) when
$G$ is equipped with the lower uniformity ${\mathcal U}_l\wedge{\mathcal U}_r$,
sometimes called  the Roelcke uniformity. Finally,
the set $A$ is said to be {\it left neutral} if for every $V\in{\mathcal N}(e)$,
there is $U\in{\mathcal N}(e)$ such that $UA\subset AV$. It is well known that the group
$G$ is FSIN if and only if every subset of $G$ is left neutral, see
\cite{PR2}.\par 

The main result is obtained as a consequence of the following
lemmas. 

\begin{lemma}
Let $(V_1,\ldots, V_q)$ {\rm ($q\geq 2$)} be a  decreasing sequence
of  neighborhoods of $e$ in $G$ and $M\subset G$, with $V_2^2\subset V_1$, such that
  $mV_{i+1}^2\subset V_im$ for every $1\leq i<q$ and $m\in M$.
Suppose  that there are a symmetric
$U\in{\mathcal N}(e)$ and $2\leq p\leq q$ such that 
$U\subset V_p$ and $m^{-p}(mu)^p\in V_p$ for every $m\in M$ and $u\in U$.\par 
Then  $M$ is right $V_1$-thin in $G$.
\end{lemma}

\begin{proof} {\rm 
Let $u\in U$
and $m\in M$. Since $m^{-p}(mu)^p\in V_p$, we have
$$m^{-(p-1)}u(mu)^{p-2}m=m^{-p}(mu)^pu^{-1}\in V_p^2\subset m^{-1}V_{p-1}m,$$
 hence $$m^{-(p-2)}u(mu)^{p-2}\in V_{p-1}.$$
We continue this calculation until we get $m^{-1}u(mu)\in V_2$, which gives
 $um\in mV_2^2\subset mV_1$ as claimed.}  
\end{proof}

\begin{lemma} Suppose that every Roelcke discrete subset of $G$ is left thin in
$G$.
Let $A\subset G$. Then, for every $V\in{\mathcal N}(e)$ and $n\geq 2$, there are $M\subset AV$ and a decreasing sequence 
$(V_1,\ldots,V_n)$ of arbitrary small neighborhoods of $e$ in $G$ {\rm (}e.g.
 $V_1\subset V$ and
$V_{2}^2\subset V_1${\rm )}, such that
\begin{enumerate}
\item[{\rm (a)}] $mV_{k+1}^2\subset V_k m$ for every $m\in M$ and $1\leq k<n$,
\item[{\rm (b)}] if $M$ is right  $V$-thin in $G$, then
$A$ is right $V^{3}$-thin in $G$.
\end{enumerate}
\end{lemma}

\begin{proof} {\rm 
 In addition 
 to $(V_1,\ldots,V_n)$, we will build 
a sequence $U_1$, ..., $U_{n-1}$ of neighborhoods
of $e$ and a sequence
  $M_1$, ..., $M_{n-1}$ of subsets of $G$; then we take $M=M_{n-1}$ and
use
  the sequence $(U_1,\ldots, U_{n-1})$  to ensure  the properties (a) and
(b).\par
 To begin, put $V_1=V$ and let $U_1$ be a symmetric neighborhood of
$e$  such that $U_1^{n}\subset V_1$ and $U_1^3\subset V_1$ (if $n< 3$).
It follows from Zorn's lemma that there is
a maximal set  $B_1\subset G$ which is  Roelcke-discrete with respect to $U_1$; 
in
particular $A\subset U_1B_1U_1$. 
For each $a\in A$, choose $(u_{(a,1)},b_{(a,1)},v_{(a,1)})\in U_1\times
B_1\times U_1$ such that
  $b_{(a,1)}=u_{(a,1)}av_{(a,1)}$ and put $M_1=\{av_{(a,1)}: a\in A\}$. Since
  $B_1$ is left $U_1$-thin in $G$, $M_1$ is left  $U_1^3$-thin in $G$, thus left $V_1$-thin in $G$. Choose 
$V_2\in {\mathcal {N}}(e)$,
with $V_2^2\subset V_1$ and   $V_2^{2n}\subset m^{-1}V_1m$ for every $m\in M_1$.\par 
At step 2,   choose
a symmetric  $U_2 \in {\mathcal{N}}(e)$ such that $U_2\subset U_1$ and
$U_2^3\subset V_2$. Zorn's lemma again gives 
a maximal $B_2\subset G$ which is  Roelcke-discrete with respect
to $U_2$,  for which we obtain  $M_1\subset U_2B_2U_2$. 
By the definition of  $M_1$, for each $a\in A$, there is 
$(u_{(a,2)},b_{(a,2)},v_{(a,2)})\in U_2\times B_2\times U_2$ such that
$b_{(a,2)}=u_{(a,2)}av_{(a,1)}v_{(a,2)}$. Put $M_2=\{av_{(a,1)}v_{(a,2)}: a\in
A\}$ and note   as above that
$M_2$ is left   $V_2$-thin in $G$. Choose 
$V_3\in {\mathcal {N}}(e)$,
with $V_3\subset V_2$ and   $V_3^{2n}\subset m^{-1}V_2m$ for every $m\in M_2$.\par 
This process allows us  to obtain 
sequences $(V_1,\ldots,V_n)$, $(M_1,\ldots, M_{n-1})$ and
$(U_1,\ldots,U_{n-1})$, with  the following:

\begin{enumerate}
\item $(U_1,\ldots,U_{n-1})$ is a decreasing sequence of symmetric
neighborhoods of $e$, with 
$U_1^{n}\subset V_1$;
\item for every $1\leq k<n$,  $U_k^3\subset V_{k}$;
\item for every $1\leq k<n$,  $M_{k}=\{av_{(a,1)}\cdots v_{(a,k)}: a\in A\}$,
where
$$(v_{(a,1)},\ldots,v_{(a,k)})\in U_1\times\ldots\times U_k;$$
\item  for every $1\leq k<n$,  $V_{k+1}^{2n}\subset m^{-1}V_km$ for each $m\in M_k$.
\end{enumerate}
It remains to verify that the properties (a) and (b)  
are satisfied by  the sequence $(V_1,\ldots,V_n)$ and the set $M=M_{n-1}$.
Note that from (3) (with $k=n-1$), it follows that $M\subset AV$
and $A\subset MV$, since $U_1\cdots U_{n-1}\subset
U_1^n\subset V$.\par
(a)
Let $m\in M$ and $1\leq  k<n$. There is $a\in A$ such that
$m=av_{(a,1)}\cdots v_{(a,n-1)}$ with
$(v_{(a,1)},\ldots,v_{(a,n-1)})\in U_1\times\ldots\times U_{n-1}$. In case
$1\leq
k<n-1$,  write $$ m=m_k\cdot v_{(a,k+1)}\cdots v_{(a,n-1)}$$
with $m_k\in M_k$. It follows from (1), (2) and (4) that 
$$mV_{k+1}^2m^{-1}\subset
m_kV_{k+1}^{2n}m_k^{-1}\subset V_k.$$ 
For $k=n-1$, the inclusions $mV_n^2m^{-1}\subset V_{n-1}$ 
for each $m\in M$, follow from (4).\par
(b)  This
follows immediately from the fact that $A\subset MV$.} 
\end{proof}

\section{FSIN versus SIN}  
  
  Following \cite{ROE}, a topological
group $G$ is said to be ASIN (for almost SIN), if there exists a
 neighborhood of the unit in $G$ which is left (or right) thin in $G$.
Equivalently, $G$ is ASIN if there exists a nonempty open
subset of $G$,  which is left (or right) thin in $G$ (indeed,
if $A,B\subset G$   are left thin in $G$, then
the set $AB$ is left thin in $G$).

\begin{proposition} Suppose that there are
  $p\geq 2$ and a nonempty open set $\Omega\subset G$
   such that the mapping $\Omega\ni x\to
x^p\in G$
 is left uniformly continuous. If every Roelcke-discrete subset of $G$
is left thin in $G$, then
$G$ is ASIN.
\end{proposition}

\begin{proof}{\rm As noted above, it suffices to  show that $G$ has a nonempty
open set
which is right thin in $G$. Fix  $g\in \Omega$ and choose $U\in{\mathcal N}(e)$
 such that $gU^3\subset \Omega$. Using Lemmas 2.1 and 2.2, we will prove that the open set $A=gU$ is right thin in $G$.\par
Let $V\in{\mathcal N}(e)$ with $V\subset U$. Applying      
Lemma 2.2 to
$A$ and $V$, we get a set $M\subset gUV$  and a 
sequence $(V_1=V,V_2,\ldots, V_p)$ of neighborhoods of $e$ satisfying (a) and
(b) in this lemma. Clearly, the assumption of
 Lemma 2.1  is  satisfied by $(V_1=V,V_2,\ldots, V_p)$ and
$M$.
Choose a symmetric $W\in{\mathcal N}(e)$ with $W\subset
U$ such that  $x^{-p}y^p\in V_p$ whenever $x,y\in \Omega$ and $x^{-1}y\in W$.
 Then, for all $m\in
M$ and $w\in W$,
we have  $m,mw\in \Omega$ and $m^{-1}mw\in W$, thus $m^{-p}(mw)^p\in V_p$. It
follows from Lemma 2.1 that $M$ is right $V$-thin in $G$, hence, by  
Lemma 2.2(b),  $A$
is right $V^3$-thin in $G$.}
\end{proof}

  The next two lemmas
correspond 
 respectively to Proposition 3.5 and Lemma 3.3 in \cite{BT2}.

\begin{lemma} Suppose that every  Roelcke-discrete subset of $G$ is left thin
in $G$. If
$G$ is ASIN, then $G$ is SIN.
\end{lemma}

\begin{lemma}  If $G$ is FSIN, then every  Roelcke-discrete subset of $G$ is
left thin
in $G$.
\end{lemma}

Now, we arrive at the main result of this note.

\begin{theo} Suppose that for some
$p\geq 2$,
the power map $G\ni g\to g^p\in G$ is left uniformly continuous
when restricted to some nonempty open subset of $G$. If every Roelcke-discrete
subset of $G$ is left thin in $G$, then $G$ is SIN.
\end{theo}

\begin{proof}  {\rm By Proposition 3.1, $G$ is ASIN. Then, from
Lemma 3.2, $G$ is SIN.}
\end{proof}

\begin{cor} Let $G$ be an FSIN group. If for some
$p\geq 2$,
the power map $G\ni g\to g^p\in G$ is left uniformly continuous
when restricted to some nonempty open subset of $G$, then $G$ is SIN.
\end{cor}

\begin{proof}  {\rm This follows from Theorem 3.4 and Lemma 3.3.}
\end{proof}

It is customary to say that the topological group $G$ 
is  {\it topologically
torsion} if for any $g \in G$ the sequence $(g^{n!})_{n\in\mathbb
N}$ converges to $e$ in $G$. The reader is referred to \cite{DIK} for
useful generalizations of this concept. We do not know if Corollary 3.5
remains true if $G$ is assumed to be topologically torsion. As a partial answer,
we offer the following.

\begin{proposition} Suppose that $G$ is FSIN and that there is $p\geq 2$ such
that for every $g\in G$, the
sequence $(g^{pn})_{n\in\mathbb N}$ converges to $e$ in $G$.
Then $G$ is SIN.
\end{proposition}

\begin{proof} {\rm Let  $V\in{\mathcal N}(e)$ and $A=G$.  As in the proof of
Proposition 3.1, considering
a sequence $(V_1=V,\ldots, V_p)$ in ${\mathcal N}(e)$ and $M\subset G$ (by Lemma
2.2), we can show that $M$ is right $V$-thin in $G$. Then, we conclude
from Lemma 2.2(b) that $G$ is right thin in itself. Indeed, take
a symmetric $W\in{\mathcal N}(e)$ with 
  $W^4\subset V_p$. For $g\in G$ and $u\in W$, there is
$N\in\mathbb N$
such that $g^{-pn}$, $g^{pn}$, $(gu)^{pn}$ and $(gu)^{-pn}$ belong to $W$ for
all $n\geq N$. Taking  $n\geq N+1$, we ontain
 $g^{-p}(gu)^p=g^{p(n-1)}g^{-pn}(gu)^{p(n+1)}(gu)^{-pn}\in
W^4\subset V_p$. Therefore, by Lemma 2.1, $M$ is right $V$-thin in $G$.}   
\end{proof}

\begin{cor} Let $G$ be an FSIN group which is of finite exponent. Then $G$
is SIN. 
\end{cor}

Clearly, Corollary 3.7 follows from both   Corollary 3.5 and
Proposition 3.6. It can be also deduced 
from the following result for which we will give a direct proof (based on Lemma 2.2).\par 
 For  $g\in G$, the left translation $l_g: G\to G$ is defined by
$l_g(h)=gh$.
For a given $A\subset G$, it is well known (and easy to check) that $A$ is left
thin in $G$ if
and only if the set $L(A)=\{l_g: g\in A\}$ is equicontinuous (at the unit $e$),
as a set of maps from the space $G$ to the uniform space $(G,{\mathcal U}_r)$.
In particular, $G$ is SIN if and only if $L(G)$ is equicontinuous. 
If $G$ is  FSIN, it suffices to suppose that for some $p\geq 1$, the
set $L(\{g^p:g\in G\})$ is equicontinuous, as it is clear from the next
statement.

\begin{proposition} Suppose that there is $p\geq 1$ such that the
set $\{g^p: g\in G\}$ is left thin in $G$. If every
Roelcke-discrete subset of $G$ is left thin in $G$, then $G$ is SIN.
\end{proposition}

\begin{proof} {\rm We may suppose that $p\geq 2$. For $V\in{\mathcal N}(e)$  and
$A=G$, choose  $(V_1=V,V_2,\ldots, V_p)$ and $M\subset G$ as
in Lemma 2.2. In view of Lemma 2.2(b), to conclude, it suffices  to verify that
$M$
is right $V$-thin in $G$. Take $U\in{\mathcal N}(e)$  such
that $g^{-p}Ug^p\subset V_p$  for every $g\in G$. Let $u\in U$ and $m\in M$;
starting from $m^{-(p-1)}um^{p-1}\in mV_pm^{-1}\subset V_{p-1}$ and continuing
this
process, we arrive at $m^{-1}um\in  V_1$, that is, $um\in mV$.} 
\end{proof}
 
 Recall
that the group $G$ is called periodic (or a torsion group) if every element of
$G$ is of finite  order. 
\begin{cor} Suppose that $G$ is FSIN, periodic and a Baire space. Then $G$
is SIN.
\end{cor}

\begin{proof} {\rm A standard Baire category argument gives  a nonempty
open
subset $O$ of $G$ and $p\geq 2$ such that $x^p=e$ for every $x\in O$.
Hence Corollary 3.5 applies.} 
\end{proof}

\begin{remark} {\rm Corollary 3.7 remains true if   Roelcke-discrete
subsets of $G$ are left thin (without assuming that $G$ is FSIN). This is of
course the
case 
when the Roelcke uniformity of $G$ is precompact. As a consequence, we obtain
the following statement: Every
topological group which is Roelcke precompact and of finite exponent
is  precompact (equivalently, a SIN group). Recall that Roelcke precompact
periodic groups need not be precompact (consider the group of finitely
supported permutations of an infinite set).}
\end{remark}

Now, for the convenience 
of the reader, we cite  an example of a 
topological group
of finite
exponent which is not SIN. Let $S$ and $A$ be two non trivial groups, with $A$ infinite, 
and consider the group $H=S^A$ with the pointwise product. The map $\eta: A\to
{\rm
Aut}(H)$ defined
by $\eta(a)(h)(b)=h(ba)$, $a,b\in A$, $h\in H$ is an homomorphism, where ${\rm
Aut}(H)$ stands for the 
automorphisms group of $H$ with the composition law $(f,g)\to f\circ g$. Let
$G=H\times_\eta A$ be the semi-direct
product group associated to $(A,H,\eta)$, topologized as follows: $A$ is
discrete and $H$ is equipped with the product topology, the group $S$ being
discrete.
Then,  $G$ is not SIN; indeed, the set $\{e\}\times A$ is neither left nor right
thin in $G$. If $S={\mathbb Z}_2$ and the group $A$ is  of
exponent $2$, then $G$ is of exponent $4$.\par 
\vskip 2mm

\noindent{\it A concluding comment.}
The statements of Theorem 3.4 and Proposition 3.8 are different,
 although they have some common consequences (e.g. Corollary 3.7). In fact, in a
way, they are complementary and  we propose the following
discussion to explain that.
In general,  a map $f:
(G,{\mathcal U}_l)\to
(X,{\mathcal U})$ (where $(X,{\mathcal U})$ is a uniform
space) is uniformly continuous if and only if the set $\{f_g: g\in G\}$
of left translations of $f$ is equicontinuous (at $e$). Thus,
the power map $\phi_p: (G,{\mathcal U}_l)\to (G,{\mathcal U}_l)$ is 
uniformly continuous
if and only if the set $\{(\phi_p)_g: g\in G\}$ of all left translations
of $\phi_p$ is left
equicontinuous
(i.e., when $G$
is equipped with the left uniformity). Taking in Lemma 2.1 (via Lemma 2.2)
a sequence $(V_1,\ldots, V_q)$ with an appropriate length,  
it is
quite possible to weaken the assumption in Theorem 3.4 assuming only that
the 
 set $\{(\phi_p)_{g^q}: g\in G\}$ is left
equicontinuous
for some $p\geq 2$ and $q\geq 1$.
 On the other hand, the
left thinness
of the set $A$
in Proposition 3.8  means that the set $\{(\phi_1)_{g^p}: g\in G\}$ is
right equicontinuous (i.e., when $G$ is equipped with the right uniformity).
It is again  possible here to 
assume only that the set $\{(\phi_p)_{g^q}: g\in G\}$ is
{\it right} equicontinuous for some $p\geq 1$ and $q\geq 1$.
\vskip 2mm
\noindent{\bf Acknowledgement.} The authors would like to thank the referee for
her/his valuable remarks and suggestions.

\end{document}